\documentclass[
    submission,
copyright,creativecommons,UKenglish]{eptcs}
\providecommand{\event}{ACT2021}
\pdfoutput=1
\usepackage[utf8]{inputenc}
\usepackage{scrextend}
\usepackage[UKenglish]{babel}
\usepackage{paralist}
\usepackage{amsmath}
\usepackage{amssymb}
\usepackage{multicol}
\usepackage{enumitem}
\usepackage{amsfonts}
\usepackage{amsthm}
\usepackage{mathrsfs}
\usepackage{graphicx}
\usepackage{stmaryrd}
\usepackage{mathtools}
\usepackage{hyperref}
\usepackage{microtype}
\usepackage{bbm}
\usepackage{tikz-cd}
\usepackage[numbers,sort&compress]{natbib}

\newcommand{\N}{\mathbb{N}}

\newcommand{\id}{\text{id}}
\newcommand\Define[1]{\textbf{#1}}
\newcommand{\sse}{\subseteq}

\newcommand{\opp}{\text{op}}

\newcounter{counter}
\theoremstyle{definition}
\newtheorem{theorem}[counter]{Theorem}
\newtheorem{proposition}[counter]{Proposition}
\newtheorem{lemma}[counter]{Lemma}
\newtheorem{definition}[counter]{Definition}
\newtheorem{corollary}[counter]{Corollary}
\newtheorem{example}[counter]{Example}
\newtheorem{remark}[counter]{Remark}

\title{A Categorical Construction of the Real Unit Interval}
\def\titlerunning{A Categorical Construction of the Real Unit Interval}

\author{John van de Wetering
\institute{Radboud Universiteit Nijmegen}
\institute{University of Oxford}
\email{john@vandewetering.name}}

\def\authorrunning{J.~van de Wetering}

\begin{document}

\maketitle

\begin{abstract}
	The real unit interval is the fundamental building block for many branches of mathematics like probability theory, measure theory, convex sets and homotopy theory.
    However, a priori the unit interval could be considered an arbitrary choice and one can wonder if there is some more canonical way in which the unit interval can be constructed.
    In this paper we find such a construction by using the theory of effect algebras. 
    We show that the real unit interval is the unique non-initial, non-final irreducible algebra of a particular monad on the category of bounded posets. The algebras of this monad carry an order, multiplication, addition and complement, and as such model much of the operations we need to do on probabilities.
    On a technical level, we show that both the categories of $\omega$-complete effect algebras as well as that of $\omega$-complete effect monoids are monadic over the category of bounded posets using Beck's monadicity theorem. The characterisation of the real unit interval then follows easily using a recent representation theorem for $\omega$-complete effect monoids.
\end{abstract}

\section{Introduction}

Probabilities in our world are represented by numbers in the real unit interval $[0,1]$. This leads to the usage of $[0,1]$ in all fields of mathematics that are motivated by a probabilistic view, such as probability theory and measure theory, and more indirectly the theory of convex sets and homotopy theory.
It naturally raises the question of how we can generalise $[0,1]$ or consider it abstractly using some of its properties. For instance, we could generalise it to the set of continuous functions $C(X,[0,1])$ from some topological space $X$ to $[0,1]$ in order to represent probabilities that can vary across space, or we could replace $[0,1]$ by an algebraic generalisation like an \emph{MV-algebra}~\cite{chang1958algebraic} or an \emph{effect algebra}~\cite{foulis1994effect} to study fields like probability theory or convex sets in a more abstract setting~\cite{jacobs2011probabilities,jacobs2017distances,adams2015state}.

Turning around this question we can ask if we can recover the standard real unit interval $[0,1]$ from some abstract conditions in order to get a clearer understanding of its central importance in mathematics. 
A categorical characterisation of the unit interval was given by Freyd~\cite{freyd2008algebraic}, who showed that the unit interval is the final coalgebra of an operation that consists of `glueing together' the ends of two spaces, which can be captured by the \emph{midpoint operation} $(a,b)\mapsto \frac12(a+b)$. This was later generalised to give characterisations of higher-dimensional simplices~\cite{leinster2011general}.
However, the ability to take a midpoint is not what one would usually consider a crucial feature of $[0,1]$, especially if we consider $[0,1]$ as a set of probabilities.
Instead, as argued in~\cite{first}, the operations on $[0,1]$ that have a direct correspondence to its use as a set of probabilities are the partially defined addition $a+b$ of probabilities $a$ and $b$ satisfying $a+b\leq 1$ that corresponds to the coarse-graining of independent events; the complement $a\mapsto 1-a$ corresponding to negation; the multiplication $a\cdot b$ corresponding to the conjunction of events; and the partial order $a\leq b$ to tell us which event is more likely. 
We will aim to derive the unit interval from algebraic structure mimicking these operations. 

An \emph{effect algebra} has a partially defined addition, and has a complement operation so that for every $a$ we have a unique $a^\perp$ with $a+a^\perp = 1$. Effect algebras were originally introduced as an abstraction of the set of effects in a C$^*$-algebra~\cite{foulis1994effect}.
As shown by Jenča~\cite{jenvca2015effect}, the category of effect algebras is isomorphic to the category of Eilenberg-Moore algebras of the \emph{Kalmbach monad}~\cite{kalmbach1977orthomodular} which arises from the free-forgetful adjunction between the category of \emph{bounded} posets (those posets with a minimum and maximum element) $\textbf{BPos}$ and the category of orthomodular posets.
Hence, one could `discover' effect algebras by considering the basic structures of (orthomodular) bounded posets.

The category of effect algebras is symmetric monoidal~\cite{jacobs2012coreflections} and the monoids resulting from this tensor product are called \emph{effect monoids}. Concretely, effect monoids are effect algebras that have an associative, unital, distributive multiplication operation. The unit interval $[0,1]$ is an effect monoid with its multiplication operation given by the standard multiplication.

The unit interval has one additional property that we aim to capture: when we have an infinite sequence of probabilities of which we can sum every finite subset, then we can sum the entire infinite set. Extending the partially defined addition operation to an infinitary partially defined operation results in a \emph{partially additive monoid}~\cite{arbib1980partially,ManesA1986}.
Effect algebras that are also partially additive monoids can be equivalently described as effect algebras that are \emph{$\omega$-complete}, i.e.~for which every increasing sequence has a supremum.
We show that the category of the resulting $\omega$-effect-algebras is equivalent to an Eilenberg-Moore category for some monad on $\textbf{BPos}$ (although note that we do not succeed in constructing this monad explicitly, instead relying on Beck's monadicity theorem).
The category of $\omega$-effect-algebras is also symmetric monoidal, and its monoids, the $\omega$-effect-monoids, have all the structure that we were interested in in $[0,1]$: a countable partial addition, a complement, and a multiplication. 
It turns out that these $\omega$-effect-monoids are particularly well-behaved. 
A recent representation theorem by Westerbaan, Westerbaan and the author~\cite{first} shows that each $\omega$-effect-monoid embeds into a direct sum of a Boolean algebra and the set of continuous functions from a topological space into $[0,1]$. In particular, we can show that the only irreducible $\omega$-effect-monoids, those that cannot be written as a non-trivial direct sum, are $\{0\}$, $\{0,1\}$ and the unit interval $[0,1]$. 
Each of these three possibilities gives a different view on probabilities. The effect monoid $\{0\}$ represents the \emph{inconsistent} world where $0=1$. The effect monoid $\{0,1\}$ represents the \emph{deterministic} world, where everything either holds or does not hold with certainty. And finally, $[0,1]$ gives us the \emph{probabilistic} world.
As $\{0\}$ is the final object in the category of $\omega$-effect-monoids and $\{0,1\}$ is the initial object, this establishes $[0,1]$ as the unique non-initial, non-final irreducible monoid in the Eilenberg-Moore category of a monad on \textbf{BPos}.

Hence, starting from the Kalmbach monad resulting from the free-forgetful adjunction between bounded posets and orthomodular posets, we get effect algebras. Then by considering an extension of the addition operation to a countable addition operation we find an `$\omega$-Kalmbach' monad, of which the algebras are $\omega$-effect-algebras. Then by focusing on the monoids in this category we find the unit interval as one of three basic irreducible objects.

Finally, we also show that the category of $\omega$-effect-monoids itself is monadic over $\textbf{BPos}$, so that we can also directly exhibit the unit interval as an irreducible algebra over $\textbf{BPos}$.

These results shed some light on why exactly the real unit interval is the canonical choice for describing probabilities: any structure with an order, a negation, a multiplication, and countable sums is either deterministic or is built out of copies of the real unit interval.
The results in this paper are in a sense a more categorical reformulation of the characterisation theorem of $\omega$-effect-monoids in~\cite{first}. That characterisation has been applied to the study of effectus theory~\cite{sigma}, sequential effect algebras~\cite{second}, and a reconstruction of quantum theory~\cite{reconstruction}.

\section{Preliminaries}

We start by recalling all the necessary concepts: orthomodular posets, effect algebras, $\omega$-completeness, and effect monoids. Although the technical results in this section are all known, by combining them we get our first categorical characterisation of the real unit interval, which is a new observation in itself.

\subsection{Posets, orthomodularity and effect algebras}

\begin{definition}
    A poset $P$ is \Define{bounded} when it has a minimal element $0$ and a maximal element $1$. A morphism of bounded posets is an order-preserving map $f:P\rightarrow Q$ satisfying $f(0)=0$ and $f(1)=1$. We denote the category of bounded posets by \textbf{BPos}.
\end{definition}

\begin{remark}\label{rem:poset-category}
    Call a category \emph{thin} when it has at most one morphism between each pair of objects, and \emph{skeletal} when isomorphic objects are equal. Then a poset corresponds to a small thin skeletal category, where we have a morphism $a\rightarrow b$ iff $a\leq b$. Monotone maps are then simply functors between these types of categories. 
    A bounded poset is such a category with an initial and final object, and a \textbf{BPos} morphism is a functor preserving the initial and final objects.
\end{remark}

Our next definition concerns a particular type of bounded poset that we call \emph{orthomodular}. 
Intuitively, an orthomodular poset is a poset equipped with a `negation operation' $\perp$. This satisfies the classical properties we would expect for a negation: $(a^\perp)^\perp = a$ (not not $a$ is just $a$), $a\wedge a^\perp = 0$ ($a$ and not $a$ is false), $a\leq b \implies b^\perp \leq a^\perp$ (If $a$ implies $b$, then not $b$ implies not $a$). However, it also satisfies two conditions that are perhaps less familiar. We say two elements $a$ and $b$ of the poset are \emph{orthogonal} when $a\leq b^\perp$. For such elements we expect to be able to take the `sum' of the elements, so we require that $a\vee b$ exists. Additionally, orthogonal elements can be treated `classically', meaning that the conjunction and disjunction distribute over one another as in a Boolean algebra. This is expressed by the equality $b = a^\perp \wedge (a \vee b)$ which holds for any pair of orthogonal $a$ and $b$.

Originally, orthomodularity was studied in the context of quantum logic, as the set of closed subspaces of a Hilbert space forms an orthomodular lattice. However, orthomodularity is actually a much more general concept. For instance, one can generalise the algebraic structure of the set of relations on a set $X$ into a \emph{relation algebra}~\cite{tarski1941calculus}, and one can then extract an orthomodular poset by looking at pairs of equivalence relations~\cite{harding1996decompositions}.

\begin{definition}
    An \Define{orthomodular} poset (OMP) is a bounded poset $(P,0,1)$ together with an \Define{orthocomplementation} operation mapping each element $a\in P$ to $a^\perp \in P$, satisfying the conditions below. We write $a\perp b$ and say $a$ and $b$ are \Define{orthogonal} when $a\leq b^\perp$.
    \begin{itemize}
        \item $(a^\perp)^\perp = a$, 
        \item $a\leq b \iff b^\perp \leq a^\perp$,
        \item $a\wedge a^\perp = 0$,
        \item if $a\perp b$, then $a\vee b$ exists,
        \item if $a\perp b$, then $b = a^\perp \wedge (a \vee b)$.
    \end{itemize}
    An \Define{OMP morphism} is a bounded poset homomorphism $f:P\rightarrow Q$ additionally satisfying $a\perp b \implies f(a)\perp f(b)$ and $a\perp b \implies f(a\vee b) = f(a)\vee f(b)$. We denote the category of OMPs by \textbf{OMP}.
\end{definition}

\begin{remark}
    Treating a bounded poset $P$ as a category as in Remark~\ref{rem:poset-category}, we can view the orthocomplementation as a functor $(\cdot)^\perp:P\rightarrow P^\opp$ satisfying $(\cdot)^\perp \circ (\cdot)^\perp = \id$. However, the other properties cannot be easily described in these categorical terms.
\end{remark}

Note that there is an evident forgetful functor $U:\textbf{OMP}\rightarrow\textbf{BPos}$.

\begin{theorem}[\cite{harding2004remarks}]
    The forgetful functor $U:\textbf{OMP}\rightarrow \textbf{BPos}$ has a left adjoint $K:\textbf{BPos}\rightarrow \textbf{OMP}$.
\end{theorem}

This left adjoint can be explicitly described: the orthomodular poset $K(P)$ is known as the \emph{Kalmbach extension}.
\begin{definition}
    Let $P$ be a poset. A \Define{chain} $C$ in $P$ is a totally ordered subset $C\sse P$. If $C$ is finite then we write it as $[a_1<a_2<\cdots<a_n]$. Let $K(P)$ be the set of finite chains of even length. This is a poset in the following way:
    \[[a_1<a_2<\cdots<a_{2n}] \leq [b_1<b_2<\cdots<b_{2m}] \iff \forall 1\leq i\leq n\exists 1\leq j\leq m: b_{2j-1}\leq a_{2i-1}< a_{2i}\leq b_{2j}.\]
    The minimal element of $K(P)$ is the empty set. When $P$ is bounded, it has a maximal element:  the chain $[0<1]$. It then also has an orthocomplement given by $S^\perp = S\Delta \{0,1\}$ where $\Delta$ denotes the symmetric difference of sets. For $P$ a bounded poset, $K(P)$ then turns out to be an orthomodular poset.
    If we have a morphism of bounded posets $f:P\rightarrow Q$ then this gives rise to an OMP morphism $K(f):K(P)\rightarrow K(Q)$ defined by $K(f)(S) = \Delta_{s\in S} \{f(s)\}$.
\end{definition}

A useful way to think about the partial order on the Kalmbach extension $K(P)$ is to see each chain $[a_1<a_2<\cdots < a_{2n-1}<a_{2n}]$ as a set of half-open intervals $[a_1,a_2), [a_3,a_4),\ldots, [a_{2n-1},a_{2n})$. We then have $S\leq S'$ in $K(P)$ if each interval of $S$ is contained in an interval of $S'$.

The free-forgetful adjunction between \textbf{OMP} and \textbf{BPos} gives rise to a monad $T:= U\circ K$ on \textbf{BPos}. We will refer to this as the \Define{Kalmbach monad}. It turns out that we can explicitly describe the Eilenberg-Moore algebras $\alpha:K(P)\rightarrow P$ that arise from this monad.

\begin{definition}
    An \Define{effect algebra}
    (EA)~\cite{foulis1994effect} is a set~$E$ with
    distinguished element~$0 \in E$,
    partial binary operation~$\ovee$ (called \Define{sum})
    and (total) unary operation~$a \mapsto~a^\perp$ (called \Define{complement}),
    satisfying the following axioms,
    writing~$a\perp b$ whenever~$a \ovee b$ is defined
        and defining~$1:= 0^\perp$.
\begin{itemize}
\item Commutativity: if $a\perp b$, then $b \perp a$ and $a\ovee b = b \ovee a$.
\item Zero: $a\perp 0$ and $a\ovee 0 = a$.
\item Associativity: if $a\perp b$ and $(a\ovee b)\perp c$, then
    $b\perp c$,
$a\perp (b \ovee c)$, and $(a\ovee b) \ovee c = a\ovee (b\ovee c)$.
\item The complement
    $a^\perp$ is the unique element with $a \ovee a^\perp = 1$.
\item If $a\perp 1$, then $a=0$.
\end{itemize}
For~$a,b \in E$ we write~$a \leq b$ whenever there is a~$c \in E$
    with~$a \ovee c = b$.
This turns~$E$ into a poset with minimum~$0$ and maximum~$1$.
The map~$a \mapsto a^\perp$ is an order anti-isomorphism.
Furthermore, $a \perp b$ if and only if~$a \leq b^\perp$.
An \Define{effect algebra homomorphism} is a map $f:E\rightarrow F$ satisfying $f(1)=1$, and $a\perp b\implies f(a)\perp f(b), f(a\ovee b) = f(a)\ovee f(b)$.
We denote the category of effect algebras by \textbf{EA}.
\end{definition}

\begin{remark}
Each OMP is an effect algebra where we set $a\perp b$ when $a\leq b^\perp$, and we define $a\ovee b := a\vee b$. The OMP morphisms are then precisely the effect algebra homomorphisms, hence there is a full and faithful functor $\textbf{OMP}\rightarrow \textbf{EA}$. As each effect algebra is a bounded poset there is a forgetful functor $\textbf{EA}\rightarrow\textbf{BPos}$.
\end{remark}

\begin{remark}\label{rem:effect-difference}
In an effect algebra, if~$a \leq b$, then the element~$c$ with~$a \ovee c = b$
    is unique. We denote this unique element by~$b \ominus a$. 
\end{remark}

\begin{theorem}[\cite{jenvca2015effect}]
    The Eilenberg-Moore category $\textbf{BPos}^T$ of the monad $T:=U\circ K: \textbf{BPos}\rightarrow \textbf{BPos}$ is isomorphic to the category of effect algebras.
\end{theorem}
\begin{proof}
We give a short sketch of how the construction behind this result works, as described in~\cite{jenvca2015effect}. Any effect algebra $E$ is a bounded poset. We define the algebra action $\alpha: K(E)\rightarrow E$ by $\alpha([a_1<a_2<\cdots<a_{2n-1}<a_{2n}]) = (a_2\ominus a_1)\ovee \cdots \ovee (a_{2n}\ominus a_{2n-1})$, so each effect algebra is an Eilenberg-Moore algebra of the Kalmbach monad. A morphism of effect algebras preserves $\ominus$ and $\ovee$, and hence the algebra structure, so that we have an embedding $\textbf{EA}\rightarrow \textbf{BPos}^T$.
Conversely, if we have an Eilenberg-Moore algebra $\alpha:K(P)\rightarrow P$ we can define a partial binary operation $b\ominus a$ on $P$ that is defined for $a\leq b$ and is given by $b\ominus a = 0$ when $b=a$ and otherwise $b\ominus a = \alpha([a<b])$. This makes $P$ into a \emph{D-poset}~\cite{kopka1994d} (the `D' stands for `difference'). This is a structure that is equivalent to an effect algebra. The operation $\ovee$ that makes a D-poset an effect algebra is given by $a\ovee b = 1\ominus ((1\ominus a)\ominus b)$.
\end{proof}




\subsection{Effect monoids and the real unit interval}

Properties of the category of effect algebras \textbf{EA} were studied in~\cite{jacobs2012coreflections}. Therein it was shown that this category is complete and co-complete (although in retrospect, because effect algebras form an Eilenberg-Moore category, the category is complete for entirely abstract reasons). A tensor product of effect algebras was also constructed, leading to symmetric monoidal structure. For our purposes it will suffice to describe this tensor product by its universal property.

\begin{definition}\label{def:bimorphism}
    Let $E$, $F$ and $G$ be effect algebras. A map $f:E\times F \rightarrow G$ is a \Define{bimorphism of effect algebras} when $f(1,1) = 1$ and for all $a_1\perp a_2$ in $E$ and $b_1\perp b_2$ in $F$ we have $f(a_1\ovee a_2, b_1) = f(a_1, b_1)\ovee f(a_2,b_1)$ and $f(a_1, b_1\ovee b_2) = f(a_1,b_1)\ovee f(a_1,b_2)$.
    A \Define{tensor product} of $E$ and $F$ consists of an effect algebra $G$ together with a bimorphism $f:E\times F \rightarrow G$ such that for every bimorphism $g:E\times F\rightarrow G'$ there is a unique effect algebra homomorphism $h:G\rightarrow G'$ such that $g=h\circ f$.
    We denote the tensor product of $E$ and $F$ (which is unique up to effect algebra isomorphism) by $E\otimes F$.
\end{definition}

Using this tensor product we get a notion of monoid $M\otimes M\rightarrow M$ internal to the category of effect algebras. These monoids can be described explicitly.

\begin{definition}\label{def:effectmonoid}
    An \Define{effect monoid} (EM) is an effect algebra $(M,\ovee, 0, (\ )^\perp,
    \,\cdot\,)$ with an additional (total) binary operation $\cdot$,
such that for all~$x,y,z \in M$:
\begin{itemize}
\item Unitality: $x\cdot 1 = x = 1\cdot x$.
\item Multiplication is a bimorphism: If $y\perp z$, then $x\cdot y\perp  x\cdot z$ and
    $y\cdot x\perp z\cdot x$
        with~$x\cdot (y\ovee z) = (x\cdot y) \ovee (x\cdot z)$
        and~$ (y\ovee z)\cdot x = (y\cdot x) \ovee (z\cdot x)$.
\item Associativity: $x \cdot (y\cdot z) = (x \cdot y) \cdot z$.
\end{itemize}
\end{definition}

\begin{example}
    An orthomodular poset is an effect monoid iff it is a Boolean algebra. In this case we have $a\cdot b := a\wedge b$.
\end{example}

\begin{example}
    Let $X$ be a compact Hausdorff space. Then the space $C(X,[0,1])$ of continuous functions from $X$ to $[0,1]$, equipped with pointwise defined addition and multiplication, is an effect monoid. Note that by the Gelfand-Naimark theorem we can equivalently describe this as the unit interval of a commutative unital C$^*$-algebra.
\end{example}

It is at this point that we can describe the real unit interval categorically (although not yet uniquely): it is a monoid in the Eilenberg-Moore category of the Kalmbach monad. The unit interval $[0,1]$ is obviously a bounded poset. It is an effect algebra with $a\perp b \iff a\leq 1-b$, $a\ovee b := a+b$ and $a^\perp := 1-a$. It is an effect monoid with $\cdot$ just the regular multiplication of real numbers.
We can be a bit more specific: $[0,1]$ as a monoid is \emph{irreducible}. To describe this we need the following straightforwardly verified proposition.

\begin{proposition}
    Let $\mathbf{C}$ be a category with products $\times$ and symmetric monoidal structure $\otimes$
    such that $\otimes$ \Define{distributes} over $\times$: $A\otimes (B\times C) \cong (A\otimes B) \times (A\otimes C)$ (in a suitably natural way).
    Given two monoids $A$ and $B$ (with respect to $\otimes$) we can construct a monoid map
    $m:(A\times B)\otimes (A\times B) \rightarrow A\times B$ using the isomorphisms given by the distributivity of $\otimes$, projection maps, compositions of the monoid maps of $A$ and $B$ and the universal property of the product.
\end{proposition}

\begin{definition}
    We call a monoid $M$ in a category with products and a distributive symmetric monoidal tensor product \Define{irreducible} when $M\cong M_1\times M_2$ implies that either $M_1$ or $M_2$ is final.%
    \footnote{It might actually be more natural to call such monoids `prime', while irreducibility with respect to a distributive \emph{co}product should be called `irreducible'. However, because products in the category of effect algebras are given by the Cartesian product it seems warranted in this case to call these monoids irreducible.}
\end{definition}

\begin{proposition}
    Let $M$ be an effect monoid. We call an element $p\in M$ \Define{idempotent} when $p\cdot p = p$.
    An effect monoid is irreducible iff the only idempotents are $0$ and $1$.
\end{proposition}
\begin{proof}
    The product in the category of effect algebras is the Cartesian product. Hence, a product effect monoid $M_1\times M_2$ is likewise defined using pointwise operations. In such an effect monoid $(1,0)$ and $(0,1)$ are idempotents not equal to $0$ and $1$.
    Conversely, if $M$ has an idempotent $p\neq 0,1$, then we can consider the `subalgebras' $pM:=\{a\in M~;~a\leq p\}$, and $p^\perp M$ (defined analogously). We then have $M=pM\times p^\perp M$. For the details see \cite[Corollary~23]{first}.
\end{proof}

As the only idempotents in $[0,1]$ are $0$ and $1$ we see that $[0,1]$ is indeed irreducible. Other irreducible effect monoids are the initial object $\{0,1\}$ and the final object $\{0\}$. 
However, there are also more pathological irreducible effect monoids. For instance, the `lexicographically ordered vector spaces' of~\cite{basmaster} or Example~40 of~\cite{second}.

\subsection{\texorpdfstring{$\omega$}{omega}-completeness}

The reason we got other possible `pathological' irreducible effect monoids is that we were not capturing enough of the `specialness' of the order of $[0,1]$. The missing ingredient is that increasing sequences $a_1\leq a_2\leq a_3\leq \cdots$ in $[0,1]$ have a unique least upper bound $\vee_i a_i$.

\begin{definition}
    Let $P$ be a bounded poset. We say $P$ is \Define{$\omega$-complete} when every increasing sequence $a_1\leq a_2\leq \cdots$ in $P$ has a supremum. We denote the full subcategories of \textbf{BPos}, \textbf{OMP} and \textbf{EA} consisting of $\omega$-complete posets by $\textbf{BPos}_\omega$, $\textbf{OMP}_\omega$, and $\textbf{EA}_\omega$ respectively.
\end{definition}

\begin{remark}
    $\omega$-completeness is a weaker version of the more well-known \emph{directed completeness}. This states that any directed set (i.e.~a set $S$ such that for each $a,b \in S$ there is a $c\in S$ such that $a,b\leq c$) has a supremum. $\omega$-completeness is equivalent to requiring this property for \emph{countable} directed sets (which can be further restricted to totally ordered countable directed sets). 
    When viewing a bounded poset as a category as in Remark~\ref{rem:poset-category}, the poset is directed complete iff it it has all \emph{directed limits}. It is $\omega$-complete if it has all countable directed limits (these are also known as \emph{sequential limits}).
\end{remark}

The full subcategory of $\textbf{BPos}^T$ consisting of those algebras whose underlying objects lie in $\textbf{BPos}_\omega$ is isomorphic to $\textbf{EA}_\omega$, which follows easily from the fact that the isomorphism between $\textbf{BPos}^T$ and $\textbf{EA}$ acts as the identity on the underlying posets. We will often write $\omega$-effect algebra (abbreviated to $\omega$EA) instead of $\omega$-complete effect algebra for brevity and similarly we will write $\omega$-effect monoid ($\omega$EM) for an $\omega$-complete effect monoid.\footnote{%
It would be reasonable to additionally require that the multiplication operation of an $\omega$-complete effect monoid preserves directed countable suprema. However, as shown in~\cite{first}, this holds automatically.%
}

\begin{remark}\label{rem:sigma-additive}
There is a different perspective on $\omega$-complete effect algebras that is more algebraic in nature. Namely, they are precisely those effect algebras that are \emph{$\sigma$-additive}. The definition is a bit technical (see~\cite{arbib1980partially,ManesA1986} or~\cite[Definition~2]{sigma}), but intuitively in such effect algebras we have a partial infinite sum operation where the sum of an infinite sequence $x_1,x_2,\ldots$ is defined precisely when the sum of every finite subset is defined.
\end{remark}

We have the following theorem characterising $\omega$-effect monoids.
\begin{theorem}[\cite{first}]
    Let $M$ be an $\omega$-effect monoid. Then there exists a Boolean algebra $B$ and a compact Hausdorff space $X$ such that $M$ embeds (as an effect monoid) into $B\times C(X,[0,1])$.
\end{theorem}
Note that this theorem in particular implies that all $\omega$-effect monoids are commutative. We will be primarily interested in the following consequence of the result.
\begin{corollary}[\cite{first}]
    Let $M$ be an $\omega$-effect monoid with no non-trivial zero divisors (i.e.~if $a\cdot b = 0$, then either $a=0$ or $b=0$). Then $M$ is isomorphic as an effect monoid to either $\{0\}$, $\{0,1\}$ or $[0,1]$.
\end{corollary}
We can recast this result into a form that will be more useful to us.
\begin{proposition}
    Let $M$ be an irreducible $\omega$-effect monoid. Then $M\cong \{0\}$, $M\cong \{0,1\}$ or $M\cong [0,1]$.
\end{proposition}
\begin{proof}
    Any irreducible effect monoid has no non-trivial zero divisors. See~\cite[Lemma~72]{second}.
\end{proof}

This proposition allows us now to give our first categorical characterisation of the real unit interval.
\begin{theorem}
    The real unit interval is the unique non-initial, non-final, irreducible monoid in the full subcategory of $\textbf{BPos}^T$ consisting of those algebras that have their underlying objects in $\textbf{BPos}_\omega$. Here $T$ is the Kalmbach monad arising from the free-forgetful adjunction between $\textbf{BPos}$ and $\textbf{OMP}$.
\end{theorem}
\begin{proof}
    $\textbf{BPos}^T$ is isomorphic to the category of effect algebras, and its restriction to objects of $\textbf{BPos}_\omega$ gives precisely $\textbf{EA}_\omega$. The monoids in this category are $\omega$-complete effect monoids. There are precisely three such effect monoids that are irreducible: $\{0\}$, $\{0,1\}$ and $[0,1]$. The first of these is final in $\textbf{EA}_\omega$ and the second initial. Hence, the only remaining non-initial, non-final irreducible monoid is $[0,1]$.
\end{proof}

\section{The \texorpdfstring{$\omega$}{omega}-Kalmbach extension}

The fact that we have to refer to a full subcategory of an Eilenberg-Moore category to get our characterisation is not very natural. If the adjunction between $\textbf{BPos}$ and $\textbf{OMP}$ were to restrict to $\textbf{BPos}_\omega$ and $\textbf{OMP}_\omega$, then we could present the result by referring to monoids in the category $\textbf{BPos}_\omega^T$.
However, it is not the case that the Kalmbach extension $K(P)$ of an $\omega$-complete poset $P$ is itself $\omega$-complete, and hence the adjunction does not restrict. Let us demonstrate this with an explicit counter-example.

\begin{example}
    Let $P=[0,1]$. Obviously $P$ is $\omega$-complete. Now consider the following family of chains in $K(P)$. 
    Define $S_n = [\frac{1}{2n}<\frac{1}{2n-1}<\cdots<\frac{1}{2}<1]$. It is clear from the definition of the partial order in $K(P)$ that $S_1\leq S_2\leq S_3\leq \cdots$. We will show that this sequence does not have a least upper bound in $K(P)$ and hence that $K(P)$ is not $\omega$-complete.

    Let $C=[a_1<a_2<\cdots < a_{2n-1}<a_{2n}]$ be an upper bound to all the $S_i$. As $S_{n+1}$ is a chain of length $2n+2$, there must by the pigeon-hole principle be an interval in $C$ that covers at least two intervals in $S_{n+1}$. Suppose then without loss of generality that the interval $[a_{2j-1},a_{2j})$ covers the adjacent intervals $[\frac{1}{2k-1},\frac{1}{2k})$ and $[\frac{1}{2k+1},\frac{1}{2k+2})$ for some $1\leq k \leq n-1$. Define now the chain $C' = C\cup \{\frac{1}{2k},\frac{1}{2k+1}\}$. Then the interval $[a_{2j-1},\frac{1}{2k})$ in $C'$ covers $[\frac{1}{2k-1},\frac{1}{2k})$ in $S_n$, and similarly $[\frac{1}{2k+1},a_{2j})$ covers $[\frac{1}{2k+1},\frac{1}{2k+2})$. It is then clear that $C'$ is also an upper bound of all the $S_i$. However, as $C'\leq C$ and $C$ was arbitrary this shows that the sequence $S_1\leq S_2\leq \cdots$ cannot have a least upper bound.
\end{example}

This example shows that the problem with the existence of suprema is that the lengths of the chains in a sequence can increase without bound, while any supremum must always have a finite number of elements. Naively, one might think that this problem can be fixed by considering a modified construction where we allow the chains to have countable length. Unfortunately, this also runs into problems.
The author has not succeeded in constructing the suitable generalisation explicitly.
Instead, we will construct it implicitly as a consequence of the adjoint functor theorem.

\begin{definition}
	Let $f:P\rightarrow Q$ be a monotone map between $\omega$-complete bounded posets. We say $f$ is \Define{$\omega$-normal} when it preserves the suprema of increasing sequences, 
	i.e.~when $f(\vee_i x_i) = \vee_i f(x_i)$ for a sequence $x_1\leq x_2\leq\cdots$.
	We denote by $\omega\textbf{EA}$ the `wide' subcategory of $\textbf{EA}_\omega$ containing just the maps that are $\omega$-normal.
\end{definition}

\begin{proposition}
	Both $\textbf{BPos}$ and $\omega\textbf{EA}$ are complete and the forgetful functor from $\omega\textbf{EA}$ preserves limits.
\end{proposition}
\begin{proof}
	To prove completeness it suffices to show a category has all (small) products and all equalisers. Products in both $\textbf{BPos}$ and $\omega\textbf{EA}$ are given by the Cartesian product with pointwise operations. In $\textbf{BPos}$ the equaliser of morphisms $f,g:E\rightarrow F$ is given by $X\sse E$ defined as $X:=\{x\in E~;~f(x)=g(x)\}$. The object $X$ inherits the poset structure from $E$, and because $f(0)=0=g(0)$ and $f(1)=1=g(1)$ it contains the bounds $0$ and $1$.
	Verifying that $X$ satisfies the appropriate universal property is entirely standard.
	If $f,g:E\rightarrow F$ are now morphisms in $\omega\textbf{EA}$ the equaliser is given in the same way:
	that $X$ is $\omega$-complete follows because $f$ and $g$ are $\omega$-normal,
	and as $f$ and $g$ preserve addition and complement, $X$ is also closed under addition and complement so that $X$ is also an effect algebra.

	As the underlying set for products and equalisers are the same, the forgetful functor indeed preserves all limits.
\end{proof}

\begin{proposition}\label{prop:forget-EA-left-adjoint}
	The forgetful functor $U:\omega\textbf{EA}\rightarrow \textbf{BPos}$ has a left adjoint.
\end{proposition}
\begin{proof}
	The adjoint functor theorem says a functor $G:\mathbf{D}\rightarrow\mathbf{C}$ from a locally small complete category $\mathbf{D}$ has a left adjoint when it preserves all limits and satisfies the \emph{solution set condition}: for all $P\in\mathbf{C}$ there is a set $I$ and an $I$-indexed family of morphisms $h_i:P\rightarrow G(A_i)$ such that all $h:P\rightarrow G(A)$ can be written as $h=G(f)\circ h_i$ for some $f:A_i\rightarrow A$.

	As $\omega\textbf{EA}$ is indeed locally small and complete and $U$ preserves all limits, it remains to verify the solution set condition. So fix a bounded poset $P$. We will denote by $\#(S)$ the cardinality of a set $S$.

	Let $\mathcal{W}_P$ be a set of bounded posets such that for every bounded poset $W'$ with $\#(W')\leq \text{max}(\#(P),\aleph_1)$ there is a $W\in\mathcal{W}_P$ such that $W'$ is isomorphic to $W$.
	Define now the family $\mathcal{H}_P = \{h_i\}_{i\in I}$ of all $\textbf{BPos}$ morphisms $h_i:P\rightarrow U(A_i)$ where $A_i\in\omega\textbf{EA}$ and $U(A_i)\in\mathcal{W}_P$.

	Let $h:P\rightarrow U(A)$ be an arbitrary morphism in $\textbf{BPos}$ for some $A\in\omega\textbf{EA}$. Let $B$ be the $\omega$EA generated by $h(P)$ in $A$. The cardinality of $B$ is bounded by $\text{max}(\#(P),\aleph_1)$, and hence $U(B)$ is isomorphic to some $A_i$ in $\mathcal{W}_P$. But then we can equip $A_i$ with an effect algebra structure as well so that $B$ is also isomorphic to $A_i$ as an effect algebra. Write $\phi:A_i\rightarrow A$ for the corresponding embedding of effect algebras in $\omega\textbf{EA}$ arising from the embedding $B\rightarrow A$. Then $h=U(\phi)\circ h_i$ for some $h_i:P\rightarrow U(A_i)$ in $\mathcal{H}_P$, and we are done.
\end{proof}

\begin{remark}
	One might be surprised that we get an adjunction between $\omega\textbf{EA}$ and $\textbf{BPos}$ instead of between $\omega\textbf{EA}$ and `$\omega\textbf{BPos}$'.
	We can see Remark~\ref{rem:sigma-additive} as an explanation for this: $\omega$EAs are not just EAs with a stronger order property, but instead we can see them as a different type of algebraic structure defined on the back of a bounded poset. This adjunction between $\textbf{BPos}$ and $\omega\textbf{EA}$ then equips a bounded poset with this infinitary algebraic structure.
\end{remark}

The left adjoint of the forgetful functor between $\textbf{EA}$ and $\textbf{BPos}$ is the Kalmbach extension, so we will refer to the left adjoint of $U:\omega\textbf{EA}\rightarrow \textbf{BPos}$ as the \emph{$\omega$-Kalmbach extension} $K_\omega:\textbf{BPos}\rightarrow\omega\textbf{EA}$.
The author does not know of an explicit description of $K_\omega$. For a finite bounded poset $P$ it is clear that we have $K_\omega(P)=K(P)$. However, for other bounded posets $P$ it is not easy to see what $K_\omega(P)$ should be, just that this has to generally be quite a complex object. 
Let us give an example to demonstrate why we need this complexity (in a somewhat heuristic manner). Take $P=[0,1]$ and write $\iota:[0,1]\rightarrow K_\omega([0,1])$ for its embedding in its $\omega$-Kalmbach extension. 
Denote by $B([0,1])$ the Borel sets on $[0,1]$ equipped with the standard inclusion order. Then we have a \textbf{BPos} morphism $f:[0,1]\rightarrow B([0,1])$ given by $f(0) = \emptyset$ and $f(a) = [0,a]$. As $B([0,1])$ is an $\omega$-complete Boolean algebra, it is an $\omega$-complete effect algebra, and hence there must be a unique $\omega\textbf{EA}$ morphism $\hat{f}:K_\omega([0,1])\rightarrow B([0,1])$ such that $\hat{f}\circ \iota = f$. By considering $\omega\textbf{EA}$ operations on the elements $\iota(a)$ in $K_\omega([0,1])$ for $a\in[0,1]$ and transporting these using $\hat{f}$ to $B([0,1])$ we can generate almost all of the Borel sets, so that $K_\omega([0,1])$ must itself contain structure mimicking that of the Borel sets.
Hence, even for a relatively simple totally ordered poset like $[0,1]$, its $\omega$-Kalmbach extension carries a complexity rivaling the hierarchy of Borel sets.

\subsection{Proving monadicity}

Just as $\textbf{EA}$ is the Eilenberg-Moore category for the Kalmbach monad, so is $\omega\textbf{EA}$ the Eilenberg-Moore category for the $\omega$-Kalmbach monad.
As we do not have an explicit description of $K_\omega$, we will use \emph{Beck's monadicity theorem} to prove this.

\begin{definition}
	A (co)limit in a category $\mathbf{C}$ is called \Define{absolute} when it is preserved by every functor with domain $\mathbf{C}$.
	We say a functor $G:\mathbf{D}\rightarrow \mathbf{C}$ \Define{reflects absolute coequalisers} when for every pair $f,g:A\rightarrow B$ in $\mathbf{D}$ such that $G(f),G(g)$ has an absolute coequaliser in $\mathbf{C}$, the pair $f,g$ has a coequaliser in $\mathbf{D}$ that is preserved by $G$.
\end{definition}

\begin{definition}
	We say a category $\mathbf{D}$ is \Define{monadic} over $\mathbf{C}$ when $\mathbf{D}$ is equivalent to an Eilenberg-Moore category $\mathbf{C}^T$ for some monad $T$.
\end{definition}

\begin{theorem}[{Beck's monadicity theorem, cf.~\S{}VI.7 of \cite{macLane2013categories}}]
	A category $\mathbf{D}$ is monadic over $\mathbf{C}$ iff there exists a functor $U:\mathbf{D}\rightarrow \mathbf{C}$ that reflects absolute coequalisers and has a left adjoint.
\end{theorem}

To prove our result we will need the following lemma.

\begin{lemma}\label{lem:extend-natural}
	Let $U:\mathbf{D}\rightarrow \mathbf{C}$ be a functor and and let $f,g:A\rightarrow B$ in $\mathbf{D}$ be such that there is an absolute coequaliser $q:U(B)\rightarrow Q$ of $U(f)$ and $U(g)$ in $\mathbf{C}$.
	Then for any functor $F:\mathbf{C}\rightarrow \mathbf{C}$ and natural transformation $\eta:FU\Rightarrow U$ there exists a unique map $\eta_Q:F(Q)\rightarrow Q$ in $\mathbf{C}$ making the following diagram commute:
	\begin{equation*}
	\begin{tikzcd}
		FU(B) \ar[r,"F(q)"] \ar[d,"\eta_B"]
		&
		F(Q) \ar[d,densely dotted, "\eta_Q"]
		\\
		 U(B)  \ar[r, "q"]
		 &
		  Q
		\end{tikzcd}
	\end{equation*}
	In words: the natural transformation $\eta$ can be extended to include the object $Q$ and morphism $q$.
\end{lemma}
\begin{proof}
	Consider the following diagram:
	\begin{equation}\label{eq:diagram1}
	\begin{tikzcd}
		FU(A) \ar[r,shift left=.75ex,"FU(f)"]
		  \ar[r,shift right=.75ex,swap,"FU(g)"]
		  \ar[d, "\eta_A"]
		&
		FU(B) \ar[r,"F(q)"] \ar[d,"\eta_B"]
		&
		F(Q) \ar[d,densely dotted, "\eta_Q"]
		\\
		U(A) \ar[r,shift left=.75ex,"U(f)"]
		  \ar[r,shift right=.75ex,swap,"U(g)"] 
		 & 
		 U(B)  \ar[r, "q"]
		 &
		  Q
		\end{tikzcd}
	\end{equation}
	We need to show that the morphism $\eta_Q$ exists and is the unique one making the right-hand square commute.

	The naturality of $\eta$ ensures that both squares on the left commute. Because $q$ is an absolute coequaliser of $U(f), U(g)$, we see that $F(q)$ is a coequaliser of $FU(f), FU(g)$, and hence both rows in the diagram are coequalisers.
	We claim that $q\circ \eta_B$ coequalises $FU(f),FU(g)$. 
	Indeed by using that $q\circ U(f) = q\circ U(g)$ and the naturality of $\eta$ we easily calculate
	\[q\circ \eta_B \circ FU(f) = q\circ U(f)\circ \eta_A = q\circ U(g)\circ \eta_A = q\circ \eta_B\circ FU(g).\]
	Now because $F(q)$ is the coequaliser of $FU(f),FU(g)$ we get the unique arrow $\eta_Q$ that makes the right-hand square commute.
\end{proof}

\begin{theorem}\label{thm:monadicity-EA}
	$\omega\textbf{EA}$ is equivalent to the Eilenberg-Moore category of the $\omega$-Kalmbach monad over $\textbf{BPos}$.
\end{theorem}
\begin{proof}
	We know that $U:\omega\textbf{EA}\rightarrow\textbf{BPos}$ has a left adjoint, so by Beck's monadicity theorem it suffices to show $U$ reflects absolute coequalisers.

	So let $A$ and $B$ be $\omega$EAs and let $f,g:A\rightarrow B$ be $\omega$\textbf{EA} morphisms such that $U(f),U(g)$ have an absolute coequaliser $q:U(B)\rightarrow Q$ in $\textbf{BPos}$. We need to show that there is an $\omega$EA $Q'$ and an $\omega$\textbf{EA} morphism $q':B\rightarrow Q'$ that coequalises $f,g$. Since we furthermore need $U(Q') = Q$ and $U(q') = q$ and $U$ is forgetful this boils down to 1) showing that $Q$ carries the structure of an effect algebra, 2) that $q$ preserves this structure, 3) that $Q$ is $\omega$-complete, 4) that $q$ is $\omega$-normal and 5) that $q$ is also a coequaliser in $\omega\textbf{EA}$. 

    \vspace{-0.5cm}
	\paragraph{Proving 1) and 2)} 
	Let $U_1:\omega\textbf{EA}\rightarrow\textbf{EA}$ and $U_2:\textbf{EA}\rightarrow\textbf{BPos}$ denote the evident forgetful functors. Then $U=U_2\circ U_1$, and we have $U(A)=U_2(A)$, $U(B)=U_2(B)$, $U(f)=U_2(f)$, $U(g)=U_2(g)$ as objects and morphisms in $\textbf{BPos}$. 
	Since $U_2: \textbf{EA}\rightarrow \textbf{BPos}$ is monadic (via the Kalmbach extension), we then have by the `if direction' of Beck's monadicity theorem that there is an effect algebra $Q'$ such that $U_2(Q')=Q$ and an effect algebra morphism $q':B\rightarrow Q'$ such that $U_2(q')=q$. 
	Since $U_2$ is just forgetful this means that $Q$ is an effect algebra and that $q$ preserves the effect algebra structure.

    \vspace{-0.5cm}
	\paragraph{Proving 3) and 4)}
	We adapt the arguments used in \cite[Theorem~3.1]{jenvca2019pseudo} that proved that pseudo-effect algebras are monadic over bounded posets.

	Let $P$ be a bounded poset. Define $I'(P)=\{s:\N\rightarrow P~;~s~\text{monotone}\}$. That is, $s\in S(P)$ consists of $s(1),s(2),s(3),\ldots \in P$ such that $s(1)\leq s(2)\leq s(3) \leq \ldots$, and hence consists of increasing sequences in $P$. We preorder $I(P)$ by setting $s\leq t \iff \exists N \forall i\geq N: s(i)\leq t(i)$. In words: $s\leq t$ when $s$ is smaller than $t$ pointwise, except for some finite set at the start of the sequence. Then define $I(P):= I'(P)/\sim$ where $\sim$ is the equivalence relation given by the preorder. I.e.~$I(P)$ consists of equivalence classes of sequences that are eventually equal. We see that $I(P)$ is a bounded poset
    with minimal element the (equivalence class of the) constant $0$ function, and maximal element the constant $1$ function.
	We make $I$ into a functor $I:\textbf{BPos}\rightarrow\textbf{BPos}$ by mapping $f:P\rightarrow Q$ to $I(f)(s)(i) = f(s(i))$.

	Let $E$ be an $\omega$EA\@. Define $S_E: IU(E)\rightarrow U(E)$ as $S_E(s) = \bigvee_i s(i)$. It is clear that $S_E$ is monotone and preserves bounds, and hence is a morphism in $\textbf{BPos}$. We claim that it forms a natural transformation $S:IU\Rightarrow U$ for the functors $IU,U:\omega\textbf{EA}\rightarrow\textbf{BPos}$.
	Indeed for any $\omega$-normal morphism $f:E\rightarrow F$ we calculate:
	\begin{align*}
	S_F(IU(f)(s(1)\leq s(2)\leq \ldots)) &= S_F(f(s(1))\leq f(s(2)) \leq \ldots) \\
	&= \vee_i f(s(i)) \\
    &= f(\vee_i s(i)) \\
    &= U(f)(S_E(s(1)\leq s(2)\leq\ldots))
	\end{align*}

	Now, by applying Lemma~\ref{lem:extend-natural} with $F:=I$ and $\eta:=S$ we see that we get a unique map $S_Q:I(Q)\rightarrow Q$ such that $S_Q\circ I(q) = q\circ S_B$.

	We claim that $S_Q$ assigns suprema to increasing sequences in $Q$, so that $Q$ is indeed $\omega$-complete. Once we have shown this, the commutativity of the righthand square of \eqref{eq:diagram1} shows that $q$ preserves these suprema, and thus that it is $\omega$-normal.

	To prove the claim we introduce another natural transformation $\alpha: \id\Rightarrow I$ of functors $\id, I:\textbf{Bpos}\rightarrow\textbf{BPos}$ defined as $\alpha_P(x) = (x\leq x\leq x \leq \ldots)$, i.e.~it maps each element to its `constant' increasing sequence. That $\alpha_P: P\rightarrow I(P)$ is a morphism in $\textbf{BPos}$ and that it forms a natural transformation is easily verified.
	Now, for any $\omega$EA $E$ we also easily see that $S_E\circ \alpha_{U(E)} = \id_{U(E)}$.
	We now augment the diagram \eqref{eq:diagram1} with an additional row.

	\begin{equation}\label{eq:triple-rows}
	\begin{tikzcd}
		U(A) \ar[r,shift left=.75ex,"U(f)"]
		  \ar[r,shift right=.75ex,swap,"U(g)"] 
		  \ar[d, "\alpha_{U(A)}"]
		&
		U(B) \ar[r, "q"]
			\ar[d, "\alpha_{U(B)}"]
		&
		Q \ar[d, "\alpha_Q"]
		\\
		IU(A) \ar[r,shift left=.75ex,"IU(f)"]
		  \ar[r,shift right=.75ex,swap,"IU(g)"]
		  \ar[d, "S_A"]
		&
		IU(B) \ar[r,"I(q)"] \ar[d,"S_B"]
		&
		I(Q) \ar[d, "S_Q"]
		\\
		U(A) \ar[r,shift left=.75ex,"U(f)"]
		  \ar[r,shift right=.75ex,swap,"U(g)"] 
		 & 
		 U(B)  \ar[r, "q"]
		 &
		  Q
		\end{tikzcd}
	\end{equation}
	From the naturality of $\alpha$ it follows the new squares here commute.
	Note that the left and middle vertical arrows compose to identities. Hence, we can `squash' the diagram to the following:

	\begin{equation}\label{eq:triple-row-squash}
	\begin{tikzcd}
		U(A) \ar[r,shift left=.75ex,"U(f)"]
		  \ar[r,shift right=.75ex,swap,"U(g)"] 
		  \ar[d, "\id_{U(A)}"]
		&
		U(B) \ar[r, "q"]
			\ar[d, "\id_{U(B)}"]
		&
		Q \ar[d, "S_Q\circ \alpha_Q"]
		\\
		U(A) \ar[r,shift left=.75ex,"U(f)"]
		  \ar[r,shift right=.75ex,swap,"U(g)"] 
		 & 
		 U(B)  \ar[r, "q"]
		 &
		  Q
		\end{tikzcd}
	\end{equation}
	Note that if we were to replace the rightmost arrow $S_Q\circ \alpha_Q$ in this diagram by $\id_{Q}$ that the squares would still commute. However, due to the present coequalisers the arrow making the squares commute is unique so that we must have $S_Q\circ \alpha_Q = \id_Q$.

	Now, let $x_1\leq x_2 \leq \ldots$ be an increasing sequence in $Q$, and let $s$ be its associated element in $I(Q)$. We claim that $S_Q(s) = \vee_i x_i$. 
    First, let us establish that it is indeed an upper bound of the sequence. For each $x_j$ we note that $\alpha_Q(x_j)\leq s$ (as $s(i)$ is indeed eventually bigger than $x_j$). Then apply the monotone map $S_Q$ to this inequality to get $x_j=\id_Q(x_j) = (S_Q\circ\alpha_Q)(x_j) \leq S_Q(s)$.
	Now suppose $x_i \leq y$ for all $i$ for some $y\in Q$. Then $s\leq \alpha_Q(y)$ in $I(Q)$. Again, apply $S_Q$ to this inequality: $S_Q(s) \leq (S_Q\circ \alpha_Q)(y) = \id_Q(y) = y$, so that $S_Q(s)$ is indeed the least upper bound of the $x_i$.

    \vspace{-0.4cm}
	\paragraph{Proving 5)} It remains to show that $q$ is also the coequaliser of $f$ and $g$ in $\omega\textbf{EA}$.
	Note first that $q\circ f = q\circ g$ since $U$ is the identity on the underlying set structure, so it only remains to show that $q$ has the universal property of coequalisers in $\omega\textbf{EA}$. So let $h:B\rightarrow C$ be an $\omega$\textbf{EA} morphism satisfying $h\circ f = h\circ g$. By applying $U$ to these equalities we find a unique $e:U(Q)\rightarrow U(C)$ satisfying $e\circ U(q) = U(h)$. 
	We need to show that $e$ is $\omega$-normal and that $e$ preserves the addition operation so that it is a morphism in $\omega\textbf{EA}$.
	Consider the following diagram:
	\begin{equation}\label{eq:coequaliser-diagram}
	\begin{tikzcd}
		IU(B) \ar[r, "IU(q)"] \ar[d, "S_B"]
		&
		I(Q) \ar[r, "I(e)"] \ar[d, "S_Q"]
		&
		IU(C) \ar[d, "S_C"]
		\\
		U(B) \ar[r, "U(q)"]
		&
		Q \ar[r, "e"]
		&
		U(C)
	\end{tikzcd}
	\end{equation}
	The left-hand square commutes because $q$ is $\omega$-normal, and the outer rectangle commutes because $e\circ U(q) = U(h)$ is $\omega$-normal. We can hence calculate:
	\[e\circ S_Q \circ IU(q) = e\circ U(q)\circ S_B = S_C\circ I(e)\circ IU(q).\]
	Now because $IU(q)$ is a coequaliser, it is an epimorphism, and hence we can `cancel' it from this equation to get $e\circ S_Q = S_C\circ I(e)$, which indeed shows that $e$ is $\omega$-normal.

	Proving that $e$ is an effect algebra homomorphism is done using a similar argument. It is entirely analogous to the proof given in~\cite{jenvca2019pseudo}, so we only give a brief sketch. First, construct the `interval' functor $I_2:\textbf{BPos}\rightarrow\textbf{BPos}$ by $I_2(P) =\{(x,y)\in P^2~;~ x\leq y\}$ and order it by $[x_1\leq y_1] \leq [x_2\leq y_2] \iff x_2\leq x_1~\&~y_1\leq y_2$. We then construct a natural transformation $\ominus: I_2 \circ U\rightarrow U$ for the functors $I_2\circ U, U:\omega\textbf{EA}\rightarrow\textbf{BPos}$ by $\ominus_E([x\leq y]) = y\ominus x$. The naturality of $\ominus$ witnesses the fact that morphisms $f:E\rightarrow F$ preserve $\ominus$ and thus are effect algebra homomorphisms.
	We then construct a diagram like \eqref{eq:coequaliser-diagram} but with $I$ replaced by $I_2$ and $S$ replaced by $\ominus$, and use the same argument to show that $e$ must preserve $\ominus$ as well.
\end{proof}

Now that we know that $\omega$-complete effect algebras are also monadic over the bounded posets we can give a different characterisation of the real unit interval, by referring to monoids in this category.

\begin{remark}
    Effect monoids were defined with respect to the tensor product in $\textbf{EA}$, which relied on the definition of effect algebra bimorphisms (see Definition~\ref{def:bimorphism}). In $\omega\textbf{EA}$ morphisms are $\omega$-normal, so we need to modify this notion of bimorphism, which would a priori give us a different notion of `$\omega$-effect monoid'.
    However, as was shown in \cite[Theorem~45]{first}, when we have an effect monoid that is $\omega$-complete, its multiplication is automatically normal. Hence, monoids in $\textbf{EA}$ that are $\omega$-complete coincide with monoids in $\omega\textbf{EA}$ produced using the modified definition of the tensor product.
\end{remark}

\begin{theorem}
	The real unit interval is the unique non-initial, non-final, irreducible monoid in $\textbf{BPos}^T$ where $T$ is the $\omega$-Kalmbach monad.
\end{theorem}

\subsection{\texorpdfstring{$\omega$}{omega}-effect monoids as Eilenberg-Moore algebras}

In this last characterisation of the real unit interval we made reference first to a type of Eilenberg-Moore algebra, and then to monoids in that category. It turns out we can cut out the middle man and directly consider $\omega$-complete effect monoids as a type of Eilenberg-Moore algebra.

\begin{definition}
    Let $\omega\textbf{EM}$ denote the category of $\omega$-effect monoids and $\omega$-normal effect monoid homomorphisms (i.e.~$\omega$-effect algebra homomorphisms that also preserve the product operation).
\end{definition}

\begin{theorem}\label{thm:monoid-monadicity}
	$\omega\textbf{EM}$ is equivalent to the Eilenberg-Moore category of some monad over $\textbf{BPos}$.
\end{theorem}
\begin{proof}
	The proof is similar in structure to that of Theorem~\ref{thm:monadicity-EA}.
    By a proof entirely analogous to that of Proposition~\ref{prop:forget-EA-left-adjoint} we can show that the forgetful functor $U:\omega\textbf{EM}\rightarrow \textbf{BPos}$ has a left-adjoint. By Beck's monadicity theorem it then remains to show that $U$ reflects absolute coequalisers.

    So just as in the proof of Theorem~\ref{thm:monadicity-EA}, suppose $f,g:A\rightarrow B$ are $\omega\textbf{EM}$ morphisms such that $U(f),U(g)$ have an absolute coequaliser $q:U(B)\rightarrow Q$ in $\textbf{BPos}$.
    We need to show that $Q$ carries the structure of an $\omega$EM and that $q$ is an $\omega\textbf{EM}$ morphism.
    As the forgetful functor factors through $\omega\textbf{EA}$ we already know that $Q$ is an $\omega$EA and that $q$ is an $\omega$\textbf{EA} morphism. It remains to show that $Q$ also has an associative, unital and distributive multiplication which is preserved by $q$.

    \vspace{-0.3cm}
    \paragraph{Existence of a product and preservation by $q$}
    Write $P_2:\textbf{BPos}\rightarrow\textbf{BPos}$ for the product functor $P_2(A):=A\times A$. 
    For an $\omega$EM $M$ define $\mu_M:P_2(U(M))\rightarrow U(M)$ by $\mu_M(a,b)=a\cdot b$. It is easily verified that $\mu$ then forms a natural transformation $\mu:P_2\circ U\Rightarrow U$. By applying Lemma~\ref{lem:extend-natural} with $F:=P_2$ and $\eta:=\mu$ we see then that there is a unique $\textbf{Bpos}$ morphism $\mu_Q:P_2(Q)\rightarrow Q$ satisfying $\mu_Q\circ P_2(q) = q\circ \mu_B$. We define a product operation on $Q$ by $a\cdot_Q b:= \mu_q(a,b)$. The naturality condition $\mu_Q\circ P_2(q) = q\circ \mu_B$ shows that $q$ is indeed a homomorphism for this product.

    \vspace{-0.3cm}
    \paragraph{Unitality of the product}
    For $A\in\textbf{BPos}$ define $\epsilon_A: A\rightarrow P_2(A)$ by $\epsilon_A(0) = (0,0)$ and $\epsilon_A(a) = (a,1)$ for $a\neq 0$. It is straightforward to check that $\epsilon_A$ is indeed a morphism in $\textbf{Bpos}$. Furthermore, it defines a natural transformation $\epsilon:\id\Rightarrow P_2$.
    Note that for any $\omega$EM $M$ we have $\mu_M\circ \epsilon_{U(M)} = \id_{U(M)}$ as $\mu_M(\epsilon_{U(M)}(a)) = \mu_M(a,1) = a\cdot 1 = a$ for $a\neq 0$ and similarly for $a=0$.
    We can then construct diagrams analogous to~\eqref{eq:triple-rows} and~\eqref{eq:triple-row-squash}, but with $\epsilon$ and $\mu$ instead of $\alpha$ and $S$ to show that $\mu_Q\circ\epsilon_Q = \id_Q$, and hence $a\cdot_Q 1 = a$. To show $1\cdot_Q a = a$ we can use an entirely analogous argument.

    \vspace{-0.3cm}
    \paragraph{Associativity of the product}
    Let $P_3:\textbf{Bpos}\rightarrow \textbf{BPos}$ denote the product functor $P_3(A):=A\times A\times A$.
    Then $\mu\times \id:P_3\circ U \Rightarrow P_2 U$ forms a natural transformation given by $(a,b,c)\mapsto (a\cdot b, c)$. Composing this with the natural transformation $\mu:P_2\circ U\Rightarrow U$ we get 
    $\mu^L:=\mu\circ (\mu\times \id):P_3\circ U \Rightarrow U$ given by $(a,b,c)\mapsto (a\cdot b) \cdot c = a\cdot b\cdot c$.
    Defining the natural transformation $\id\times \mu$ analogously we then see that $\mu^R = \mu\circ (\id\times \mu)  = \mu\circ (\mu\times \id) = \mu^L$, which expresses the associativity of the product in an effect monoid.
    Now, take the diagram~\eqref{eq:triple-rows} but with $S$ replaced by $\mu$ and $\alpha$ replaced by $\mu\times \id$ and `squash' the columns like in diagram~\eqref{eq:triple-row-squash}. Do the same but instead with $\alpha$ replaced by $\id\times \mu$. We then have the following pair of diagrams:
    \begin{equation}\label{eq:diagrams-equality}
        \begin{tikzcd}
        U(A) \ar[r,shift left=.75ex,"U(f)"]
          \ar[r,shift right=.75ex,swap,"U(g)"] 
          \ar[d, "\mu^L_A"]
        &
        U(B) \ar[r, "q"]
            \ar[d, "\mu^L_B"]
        &
        Q \ar[d, "\mu^L_Q"]
        \\
        U(A) \ar[r,shift left=.75ex,"U(f)"]
          \ar[r,shift right=.75ex,swap,"U(g)"] 
         & 
         U(B)  \ar[r, "q"]
         &
          Q
    \end{tikzcd}
    \qquad\qquad
        \begin{tikzcd}
        U(A) \ar[r,shift left=.75ex,"U(f)"]
          \ar[r,shift right=.75ex,swap,"U(g)"] 
          \ar[d, "\mu^R_A"]
        &
        U(B) \ar[r, "q"]
            \ar[d, "\mu^R_B"]
        &
        Q \ar[d, "\mu^R_Q"]
        \\
        U(A) \ar[r,shift left=.75ex,"U(f)"]
          \ar[r,shift right=.75ex,swap,"U(g)"] 
         & 
         U(B)  \ar[r, "q"]
         &
          Q
    \end{tikzcd}
\end{equation}
The arrows $\mu^L_Q$ and $\mu^R_Q$ in their respective diagrams are the unique ones making the square commute due to the presence of the coequalisers. However, as $\mu^L_A = \mu^R_A$ and $\mu^L_B=\mu^R_B$, we see that that we can replace $\mu^R_Q$ in the diagram on the right by $\mu^L_Q$ and still retain this commutative square. Hence, by uniqueness $\mu^L_Q = \mu^R_Q$, so that the product $\cdot_Q$ is indeed associative.

\vspace{-0.3cm}
\paragraph{Distributivity of the product} It remains to show that $a\cdot_Q(b\ovee c) = (a\cdot_Q b) \ovee (a\cdot_Q c)$ (the case for $(b\ovee c)\cdot_Q a$ follows by symmetry). As it is more straightforward, we will prove the equivalent statement $a\cdot_Q (c\ominus b) = (a\cdot_Q c)\ominus (a\cdot_Q b)$ for all $a$ and $b\leq c$ in $Q$.
We do this by constructing some appropriate functors and natural transformations.
As at the end of the proof of Theorem~\ref{thm:monadicity-EA} we require the interval functor $I_2:\textbf{BPos}\rightarrow\textbf{BPos}$ defined by $I_2(P) =\{(a,b)\in P^2~;~ a\leq b\}$ which is ordered by $[a_1\leq b_1] \leq [a_2\leq b_2] \iff a_2\leq a_1~\&~b_1\leq b_2$. 
Then there is a natural transformation $\ominus: I_2 \circ U\rightarrow U$ defined by $\ominus_M([b\leq c]) = c\ominus b$. This natural transformation extends to $Q$ by Lemma~\ref{lem:extend-natural}.
We also have a natural transformation $\beta:(\id\times I_2)\circ(P_2\circ U) \Rightarrow I_2\circ U$ given by morphisms $\beta_A: A\times I_2(A)\rightarrow I_2(A)$ defined by $\beta_A(a, [b\leq c]) = [a\cdot b \leq a\cdot c]$.
The distributivity of the product over addition is then witnessed by the equation of natural transformations $\ominus\circ \beta = \mu\circ (\id\times \ominus)$.
By building an analogue of the diagrams~\eqref{eq:diagrams-equality} we can extend this equality to those morphisms defined for $Q$, so that multiplication is also distributive in $Q$.

\vspace{-0.3cm}
\paragraph{$q$ is a coequaliser} We now know that $Q$ is an $\omega$EM and that $q$ is an $\omega$\textbf{EM} morphism. It remains to show that $q$ is also a coequaliser in $\omega$\textbf{EM}. This is done in the same way as in the proof of Theorem~\ref{thm:monadicity-EA} (i.e.~we need to show that the unique morphisms that exist due to the universal property of $q$ in \textbf{BPos} are actually $\omega$\textbf{EM} morphisms; this is done by constructing some appropriate commutative diagrams).
\end{proof}

\begin{corollary}
	There exists a monad $T$ on \textbf{BPos} such that the real unit interval is the unique non-initial, non-final irreducible algebra in $\textbf{BPos}^T\cong \omega\textbf{EM}$.
\end{corollary}

\section{Conclusion}

We looked at the theory of $\omega$-complete effect monoids and showed how it can be used to give a categorical characterisation of the real unit interval. These effect monoids have an order, addition, complement, and multiplication, so this shows that this algebraic structure suffices to reconstruct the unit interval. 

An interesting aspect of our characterisation is that the unit interval is the unique irreducible monoid (in a suitable category) that is not initial nor final. This seems to say that while it is a simple structure (because it is irreducible), it is not \emph{too} simple (initial or final).

For future work it would be interesting to give a concrete description of the $\omega$-Kalmbach monad. While such a concrete description might prove difficult for general bounded posets, it might be easier for totally ordered sets like the real unit interval, because the description of the regular Kalmbach extension for those sets is also simpler (it is essentially the free Boolean algebra generated by its downsets), so perhaps we get the $\omega$-Kalmbach extension by looking at free $\omega$-Boolean algebras. A perhaps easier question to answer is whether the $\omega$-Kalmbach extension of a bounded poset always forms an orthomodular poset itself, and whether the monad arises from an adjunction with the category of $\omega$-complete orthomodular posets.

In the last couple of years there has been significant progress in the study of synthetic probability theory, for instance through the use of Markov categories~\cite{fritzSyntheticApproachMarkov2020,fritz2020representable,Jacobs2020DeFinetti,cho2017disintegration}. It would be interesting to see how the work in this paper can be related to those frameworks.

\paragraph{Acknowledgements}
The author would like to thank Bas and Bram Westerbaan for insightful discussions.
The author is supported by a Rubicon fellowship financed by the Dutch Research Council (NWO).

\bibliographystyle{eptcs}
\bibliography{main}

\begin{thebibliography}{10}
\providecommand{\bibitemdeclare}[2]{}
\providecommand{\surnamestart}{}
\providecommand{\surnameend}{}
\providecommand{\urlprefix}{Available at }
\providecommand{\url}[1]{\texttt{#1}}
\providecommand{\href}[2]{\texttt{#2}}
\providecommand{\urlalt}[2]{\href{#1}{#2}}
\providecommand{\doi}[1]{doi:\urlalt{http://dx.doi.org/#1}{#1}}
\providecommand{\bibinfo}[2]{#2}

\bibitemdeclare{article}{adams2015state}
\bibitem{adams2015state}
\bibinfo{author}{Robin \surnamestart Adams\surnameend} \& \bibinfo{author}{Bart
  \surnamestart Jacobs\surnameend} (\bibinfo{year}{2015}):
  \emph{\bibinfo{title}{State and Effect Logics for Deterministic,
  Non-deterministic, Probabilistic and Quantum Computation}}.
\newblock {\sl \bibinfo{journal}{TYPES 2015}}, p.~\bibinfo{pages}{8}.

\bibitemdeclare{article}{arbib1980partially}
\bibitem{arbib1980partially}
\bibinfo{author}{Michael~A \surnamestart Arbib\surnameend} \&
  \bibinfo{author}{Ernest~G \surnamestart Manes\surnameend}
  (\bibinfo{year}{1980}): \emph{\bibinfo{title}{Partially Additive Categories
  and Flow-Diagram Semantics}}.
\newblock {\sl \bibinfo{journal}{Journal of Algebra}}
  \bibinfo{volume}{62}(\bibinfo{number}{1}), pp. \bibinfo{pages}{203--227}.

\bibitemdeclare{article}{chang1958algebraic}
\bibitem{chang1958algebraic}
\bibinfo{author}{Chen~Chung \surnamestart Chang\surnameend}
  (\bibinfo{year}{1958}): \emph{\bibinfo{title}{Algebraic analysis of many
  valued logics}}.
\newblock {\sl \bibinfo{journal}{Transactions of the American Mathematical
  society}} \bibinfo{volume}{88}(\bibinfo{number}{2}), pp.
  \bibinfo{pages}{467--490}.

\bibitemdeclare{article}{cho2017disintegration}
\bibitem{cho2017disintegration}
\bibinfo{author}{Kenta \surnamestart Cho\surnameend} \& \bibinfo{author}{Bart
  \surnamestart Jacobs\surnameend} (\bibinfo{year}{2019}):
  \emph{\bibinfo{title}{Disintegration and {Bayesian} inversion via string
  diagrams}}.
\newblock {\sl \bibinfo{journal}{Mathematical Structures in Computer Science}}
  \bibinfo{volume}{29}(\bibinfo{number}{7}), p. \bibinfo{pages}{938–971},
  \doi{10.1017/S0960129518000488}.

\bibitemdeclare{inproceedings}{sigma}
\bibitem{sigma}
\bibinfo{author}{Kenta \surnamestart Cho\surnameend}, \bibinfo{author}{Bas
  \surnamestart Westerbaan\surnameend} \& \bibinfo{author}{John \surnamestart
  van~de Wetering\surnameend} (\bibinfo{year}{2021}):
  \emph{\bibinfo{title}{{Dichotomy between Deterministic and Probabilistic
  Models in Countably Additive Effectus Theory}}}.
\newblock In \bibinfo{editor}{Beno\^it \surnamestart Valiron\surnameend},
  \bibinfo{editor}{Shane \surnamestart Mansfield\surnameend},
  \bibinfo{editor}{Pablo \surnamestart Arrighi\surnameend} \&
  \bibinfo{editor}{Prakash \surnamestart Panangaden\surnameend}, editors: {\sl
  \bibinfo{booktitle}{Proceedings 17th International Conference on Quantum
  Physics and Logic, Paris, France, June 2 - 6, 2020}}, {\sl
  \bibinfo{series}{Electronic Proceedings in Theoretical Computer Science}}
  \bibinfo{volume}{340}, \bibinfo{publisher}{Open Publishing Association}, pp.
  \bibinfo{pages}{91--113}, \doi{10.4204/EPTCS.340.5}.

\bibitemdeclare{article}{foulis1994effect}
\bibitem{foulis1994effect}
\bibinfo{author}{David~J \surnamestart Foulis\surnameend} \&
  \bibinfo{author}{Mary~K \surnamestart Bennett\surnameend}
  (\bibinfo{year}{1994}): \emph{\bibinfo{title}{Effect Algebras and Unsharp
  Quantum Logics}}.
\newblock {\sl \bibinfo{journal}{Foundations of physics}}
  \bibinfo{volume}{24}(\bibinfo{number}{10}), pp. \bibinfo{pages}{1331--1352},
  \doi{10.1007/BF02283036}.

\bibitemdeclare{article}{freyd2008algebraic}
\bibitem{freyd2008algebraic}
\bibinfo{author}{Peter \surnamestart Freyd\surnameend} (\bibinfo{year}{2008}):
  \emph{\bibinfo{title}{Algebraic real analysis.}}
\newblock {\sl \bibinfo{journal}{Theory and Applications of Categories}}
  \bibinfo{volume}{20}, pp. \bibinfo{pages}{215--306}.

\bibitemdeclare{article}{fritzSyntheticApproachMarkov2020}
\bibitem{fritzSyntheticApproachMarkov2020}
\bibinfo{author}{Tobias \surnamestart Fritz\surnameend} (\bibinfo{year}{2020}):
  \emph{\bibinfo{title}{A synthetic approach to {Markov} kernels, conditional
  independence and theorems on sufficient statistics}}.
\newblock {\sl \bibinfo{journal}{Advances in Mathematics}}
  \bibinfo{volume}{370}, \doi{10.1016/j.aim.2020.107239}.

\bibitemdeclare{article}{fritz2020representable}
\bibitem{fritz2020representable}
\bibinfo{author}{Tobias \surnamestart Fritz\surnameend},
  \bibinfo{author}{Tom{\'a}{\v{s}} \surnamestart Gonda\surnameend},
  \bibinfo{author}{Paolo \surnamestart Perrone\surnameend} \&
  \bibinfo{author}{Eigil~Fjeldgren \surnamestart Rischel\surnameend}
  (\bibinfo{year}{2020}): \emph{\bibinfo{title}{{Representable Markov
  Categories and Comparison of Statistical Experiments in Categorical
  Probability}}}.
\newblock {\sl \bibinfo{journal}{arXiv preprint arXiv:2010.07416}},
  \doi{10.48550/arXiv.2010.07416}.

\bibitemdeclare{article}{harding1996decompositions}
\bibitem{harding1996decompositions}
\bibinfo{author}{John \surnamestart Harding\surnameend} (\bibinfo{year}{1996}):
  \emph{\bibinfo{title}{{Decompositions in Quantum Logic}}}.
\newblock {\sl \bibinfo{journal}{Transactions of the American Mathematical
  Society}} \bibinfo{volume}{348}(\bibinfo{number}{5}), pp.
  \bibinfo{pages}{1839--1862}, \doi{10.1090/S0002-9947-96-01548-6}.

\bibitemdeclare{article}{harding2004remarks}
\bibitem{harding2004remarks}
\bibinfo{author}{John \surnamestart Harding\surnameend} (\bibinfo{year}{2004}):
  \emph{\bibinfo{title}{Remarks on Concrete Orthomodular Lattices}}.
\newblock {\sl \bibinfo{journal}{International Journal of Theoretical Physics}}
  \bibinfo{volume}{43}(\bibinfo{number}{10}), pp. \bibinfo{pages}{2149--2168},
  \doi{10.1023/B:IJTP.0000049016.83846.72}.

\bibitemdeclare{article}{jacobs2011probabilities}
\bibitem{jacobs2011probabilities}
\bibinfo{author}{Bart \surnamestart Jacobs\surnameend} (\bibinfo{year}{2011}):
  \emph{\bibinfo{title}{Probabilities, Distribution Monads, and Convex
  Categories}}.
\newblock {\sl \bibinfo{journal}{Theoretical Computer Science}}
  \bibinfo{volume}{412}(\bibinfo{number}{28}), pp. \bibinfo{pages}{3323--3336},
  \doi{10.1016/j.tcs.2011.04.005}.

\bibitemdeclare{article}{jacobs2012coreflections}
\bibitem{jacobs2012coreflections}
\bibinfo{author}{Bart \surnamestart Jacobs\surnameend} \&
  \bibinfo{author}{Jorik \surnamestart Mandemaker\surnameend}
  (\bibinfo{year}{2012}): \emph{\bibinfo{title}{Coreflections in Algebraic
  Quantum Logic}}.
\newblock {\sl \bibinfo{journal}{Foundations of physics}}
  \bibinfo{volume}{42}(\bibinfo{number}{7}), pp. \bibinfo{pages}{932--958},
  \doi{10.1007/s10701-012-9654-8}.

\bibitemdeclare{inproceedings}{Jacobs2020DeFinetti}
\bibitem{Jacobs2020DeFinetti}
\bibinfo{author}{Bart \surnamestart Jacobs\surnameend} \& \bibinfo{author}{Sam
  \surnamestart Staton\surnameend} (\bibinfo{year}{2020}):
  \emph{\bibinfo{title}{De {Finetti}'s {Construction} as a {Categorical}
  {Limit}}}.
\newblock In \bibinfo{editor}{Daniela \surnamestart Petrisan\surnameend} \&
  \bibinfo{editor}{Jurriaan \surnamestart Rot\surnameend}, editors: {\sl
  \bibinfo{booktitle}{Coalgebraic {Methods} in {Computer} {Science}}},
  \bibinfo{publisher}{Springer International Publishing}, pp.
  \bibinfo{pages}{90--111}, \doi{10.1007/978-3-030-57201-3\_6}.

\bibitemdeclare{article}{jacobs2017distances}
\bibitem{jacobs2017distances}
\bibinfo{author}{Bart \surnamestart Jacobs\surnameend} \&
  \bibinfo{author}{Abraham \surnamestart Westerbaan\surnameend}
  (\bibinfo{year}{2020}): \emph{\bibinfo{title}{{Distances between States and
  between Predicates}}}.
\newblock {\sl \bibinfo{journal}{{Logical Methods in Computer Science}}}
  \bibinfo{volume}{16}(\bibinfo{number}{1}), \doi{10.23638/LMCS-16(1:26)2020}.

\bibitemdeclare{article}{jenvca2015effect}
\bibitem{jenvca2015effect}
\bibinfo{author}{Gejza \surnamestart Jen{\v{c}}a\surnameend}
  (\bibinfo{year}{2015}): \emph{\bibinfo{title}{Effect Algebras are the
  {Eilenberg-Moore} Category for the {Kalmbach} Monad}}.
\newblock {\sl \bibinfo{journal}{Order}}
  \bibinfo{volume}{32}(\bibinfo{number}{3}), pp. \bibinfo{pages}{439--448},
  \doi{10.1007/s11083-014-9344-6}.

\bibitemdeclare{article}{jenvca2019pseudo}
\bibitem{jenvca2019pseudo}
\bibinfo{author}{Gejza \surnamestart Jen{\v{c}}a\surnameend}
  (\bibinfo{year}{2019}): \emph{\bibinfo{title}{Pseudo effect algebras are
  algebras over bounded posets}}.
\newblock {\sl \bibinfo{journal}{Fuzzy Sets and Systems}},
  \doi{10.1016/j.fss.2019.07.003}.

\bibitemdeclare{article}{kalmbach1977orthomodular}
\bibitem{kalmbach1977orthomodular}
\bibinfo{author}{Gudrun \surnamestart Kalmbach\surnameend}
  (\bibinfo{year}{1977}): \emph{\bibinfo{title}{Orthomodular Lattices Do Not
  Satisfy Any Special Lattice Equation}}.
\newblock {\sl \bibinfo{journal}{Archiv der Mathematik}}
  \bibinfo{volume}{28}(\bibinfo{number}{1}), pp. \bibinfo{pages}{7--8}.

\bibitemdeclare{article}{kopka1994d}
\bibitem{kopka1994d}
\bibinfo{author}{Franti{\v{s}}ek \surnamestart K{\^o}pka\surnameend} \&
  \bibinfo{author}{Ferdinand \surnamestart Chovanec\surnameend}
  (\bibinfo{year}{1994}): \emph{\bibinfo{title}{$D$-posets}}.
\newblock {\sl \bibinfo{journal}{Mathematica Slovaca}}
  \bibinfo{volume}{44}(\bibinfo{number}{1}), pp. \bibinfo{pages}{21--34}.

\bibitemdeclare{article}{leinster2011general}
\bibitem{leinster2011general}
\bibinfo{author}{Tom \surnamestart Leinster\surnameend} (\bibinfo{year}{2011}):
  \emph{\bibinfo{title}{A general theory of self-similarity}}.
\newblock {\sl \bibinfo{journal}{Advances in Mathematics}}
  \bibinfo{volume}{226}(\bibinfo{number}{4}), pp. \bibinfo{pages}{2935--3017},
  \doi{10.1016/j.aim.2010.10.009}.

\bibitemdeclare{book}{macLane2013categories}
\bibitem{macLane2013categories}
\bibinfo{author}{Saunders \surnamestart MacLane\surnameend}
  (\bibinfo{year}{1971}): \emph{\bibinfo{title}{Categories for the working
  mathematician}}.
\newblock \bibinfo{publisher}{Springer-Verlag New York},
  \doi{10.1007/978-1-4612-9839-7}.

\bibitemdeclare{book}{ManesA1986}
\bibitem{ManesA1986}
\bibinfo{author}{Ernest~G. \surnamestart Manes\surnameend} \&
  \bibinfo{author}{Michael~A. \surnamestart Arbib\surnameend}
  (\bibinfo{year}{1986}): \emph{\bibinfo{title}{Algebraic Approaches to Program
  Semantics}}.
\newblock \bibinfo{series}{Monographs in Computer Science},
  \bibinfo{publisher}{Springer}, \doi{10.1007/978-1-4612-4962-7}.

\bibitemdeclare{article}{tarski1941calculus}
\bibitem{tarski1941calculus}
\bibinfo{author}{Alfred \surnamestart Tarski\surnameend}
  (\bibinfo{year}{1941}): \emph{\bibinfo{title}{On the calculus of relations}}.
\newblock {\sl \bibinfo{journal}{The Journal of Symbolic Logic}}
  \bibinfo{volume}{6}(\bibinfo{number}{3}), pp. \bibinfo{pages}{73--89}.

\bibitemdeclare{inproceedings}{first}
\bibitem{first}
\bibinfo{author}{Abraham \surnamestart Westerbaan\surnameend},
  \bibinfo{author}{Bas \surnamestart Westerbaan\surnameend} \&
  \bibinfo{author}{John \surnamestart van~de Wetering\surnameend}
  (\bibinfo{year}{2020}): \emph{\bibinfo{title}{A Characterisation of Ordered
  Abstract Probabilities}}.
\newblock In: {\sl \bibinfo{booktitle}{Proceedings of the 35th Annual ACM/IEEE
  Symposium on Logic in Computer Science}}, \bibinfo{series}{LICS ’20},
  \bibinfo{publisher}{Association for Computing Machinery},
  \bibinfo{address}{New York, NY, USA}, p. \bibinfo{pages}{944–957},
  \doi{10.1145/3373718.3394742}.

\bibitemdeclare{article}{second}
\bibitem{second}
\bibinfo{author}{Abraham \surnamestart Westerbaan\surnameend},
  \bibinfo{author}{Bas \surnamestart Westerbaan\surnameend} \&
  \bibinfo{author}{John van~de \surnamestart Wetering\surnameend}
  (\bibinfo{year}{2020}): \emph{\bibinfo{title}{The three types of normal
  sequential effect algebras}}.
\newblock {\sl \bibinfo{journal}{{Quantum}}} \bibinfo{volume}{4}, p.
  \bibinfo{pages}{378}, \doi{10.22331/q-2020-12-24-378}.

\bibitemdeclare{article}{reconstruction}
\bibitem{reconstruction}
\bibinfo{author}{Bas \surnamestart Westerbaan\surnameend} \&
  \bibinfo{author}{John \surnamestart van~de Wetering\surnameend}
  (\bibinfo{year}{2021}): \emph{\bibinfo{title}{{A computer scientist's
  reconstruction of quantum theory}}}.
\newblock {\sl \bibinfo{journal}{arXiv preprint arXiv:2109.10707}},
  \doi{10.48550/arXiv.2109.10707}.

\bibitemdeclare{mastersthesis}{basmaster}
\bibitem{basmaster}
\bibinfo{author}{Bas~E \surnamestart Westerbaan\surnameend}
  (\bibinfo{year}{2013}): \emph{\bibinfo{title}{Sequential Product on Effect
  Logics}}.
\newblock Master's thesis, \bibinfo{school}{Radboud University Nijmegen}.
\newblock \bibinfo{note}{Available at
  \url{https://www.ru.nl/publish/pages/813276/masterscriptie_bas_westerbaan.pdf}}.

\end{thebibliography}

\end{document}